\documentclass[12pt]{article}

\usepackage{amsmath, epsfig, cite}
\usepackage{amssymb}
\usepackage{amsfonts}
\usepackage{latexsym}
\usepackage{amsthm}

\makeatletter
\renewcommand{\@seccntformat}[1]{{\csname the#1\endcsname}.\hspace{.5em}}
\makeatother

\newtheorem{thm}{Theorem}[section]

\newtheorem{conj}[thm]{Conjecture}
\newtheorem{lem}[thm]{Lemma}

\numberwithin{equation}{section}

\begin{document}

\renewcommand{\thefootnote}{*}

\begin{center}
{\Large\bf Some variations of a ``divergent" Ramanujan-type\\[5pt] $q$-supercongruence}
\end{center}

\vskip 2mm \centerline{Victor J. W. Guo  }
\begin{center}
{\footnotesize School of Mathematical Sciences, Huaiyin Normal University, Huai'an 223300, Jiangsu,\\
 People's Republic of China\\
{\tt jwguo@hytc.edu.cn } }
\end{center}

\vskip 0.7cm \noindent{\small{\bf Abstract.}} Using the $q$-Wilf--Zeilberger method and a $q$-analogue of a ``divergent" Ramanujan-type supercongruence,
we give several $q$-supercongruences modulo the fourth power of a cyclotomic polynomial. One of them is a $q$-analogue of
a supercongruence recently proved by Wang: for any prime $p>3$,
$$
\sum_{k=0}^{p-1} (3k-1)\frac{(\frac{1}{2})_k (-\frac{1}{2})_k^2 }{k!^3}4^k\equiv p-2p^3 \pmod{p^4},
$$
where $(a)_k=a(a+1)\cdots (a+k-1)$ is the Pochhammer symbol.



\section{Introduction}
By making use of the Wilf--Zeilberger (abbr. WZ) method \cite{WZ1,WZ2}, Guillera and Zudilin \cite{GuZu} established the following supercongruence:
for any odd prime $p$,
\begin{align}
\sum_{k=0}^{(p-1)/2} \frac{3k+1}{16^k}{2k\choose k}^3\equiv p\pmod{p^3}, \label{eq:div-1}
\end{align}
We can also sum $k$ in \eqref{eq:div-1} up to $p-1$, since
the $p$-adic order of $(\frac{1}{2})_k/k!$ is $1$ for $k$ in the range
$(p+1)/2\leqslant k\leqslant p-1$.
In the sprit of \cite{Zudilin}, the supercongruence \eqref{eq:div-1} corresponds to a divergent Ramanujan-type series for $1/\pi$:
\begin{align}
\sum_{k=0}^{\infty} \frac{3k+1}{16^k}{2k\choose k}^3\ \text{``="}\ \frac{-2i}{\pi}  \label{eq:div-2}
\end{align}
(see \cite[(47)]{GuZu}). Here the summation in \eqref{eq:div-2} must be understood as the analytic continuation of the corresponding hypergeometric series.

Still using the WZ method and the divisibility result: for $n>1$,
$$
2n{2n\choose n}\Bigg| \sum_{k=0}^{n-1} (3k+1){2k\choose k}^3 16^{n-k-1},
$$
which was conjectured by Z.-W. Sun \cite{Sun4} and confirmed by Mao and Zhang \cite{MZ}, B.Y. Sun \cite{SunBY} proved the following result:
for $n>1$,
\begin{equation}
2n{2n\choose n}\Bigg| \sum_{k=0}^{n-1} \frac{6k^4}{2k-1}{2k\choose k}^316^{n-k-1}.  \label{eq:sunby}
\end{equation}
Motivated by B.Y. Sun's work, we found the following supercongruence: for any prime $p>3$,
\begin{align}
\sum_{k=0}^{p-1} \frac{6k^4}{16^k(2k-1)}{2k\choose k}^3\equiv p+2p^3 \pmod{p^4}. \label{eq:guo-1}
\end{align}

We shall prove the supercongruence \eqref{eq:guo-1} by establishing its $q$-analogue. Recall that the {\it $q$-shifted factorial} is defined by $(a;q)_0=1$ and $(a;q)_n=(1-a)(1-aq)\cdots (1-aq^{n-1})$ for $n\geqslant 1$,
and the {\it $q$-integer} is defined as $[n]=[n]_q=(1-q^n)/(1-q)$ (see \cite{GR}). Moreover, the $n$-th {\it cyclotomic polynomial}  $\Phi_n(q)$ is given by
\begin{align*}
\Phi_n(q):=\prod_{\substack{1\leqslant k\leqslant n\\ \gcd(k,n)=1}}(q-\zeta^k),  
\end{align*}
where $\zeta$ is an $n$-th primitive root of unity.
Our $q$-analogue of \eqref{eq:guo-1} can be stated as follows.

\begin{thm}\label{thm:main-1}
Let $n>1$ be an odd integer. Then
\begin{align}
&\sum_{k=0}^{n-1}\frac{[3k][2k][k]^2}{[2k-1](-q;q)_k^4}{2k\brack k}^3 q^{-(k^2+3k)/2} \notag\\[5pt]
&\quad\equiv [n]q^{-(n+1)/2}+(1+q)[n]^3+\frac{(n^2-1)(1-q)^2}{24}[n]^3q^{-(n+1)/2}  \pmod{[n]\Phi_n(q)^3},  \label{eq:guo-q}
\end{align}
where ${2k\brack k}=(q;q)_{2k}/(q;q)_k^2$ denotes the central $q$-binomial coefficient.
\end{thm}

Here we say that two rational functions $A(q)$ and $B(q)$ are congruent
modulo a polynomial $P(q)$ if and only if $P(q)$ divides the numerator of the reduced form of $A(q)-B(q)$
in the polynomial ring $\mathbb{Z}[q]$.

It is easy to see that, letting $n=p>3$ be a prime and taking $q\to 1$ in \eqref{eq:guo-q}, we obtain the supercongruence \eqref{eq:guo-1}.
Furthermore, we can also deduce from \eqref{eq:guo-q} that, for any prime $p>3$ and $r\geqslant 2$,
\begin{align*}
\sum_{k=0}^{p^r-1} \frac{6k^4}{16^k(2k-1)}{2k\choose k}^3\equiv p^r \pmod{p^{r+3} }.
\end{align*}

Recently, via the WZ method and the summation package {\tt Sigma} \cite{Schneider}, Wang \cite{Wang} proved the following supercongruence:
for any prime $p>3$,
\begin{align}
\sum_{k=0}^{p-1} (3k-1)\frac{(\frac{1}{2})_k (-\frac{1}{2})_k^2 }{k!^3}4^k\equiv p-2p^3 \pmod{p^4}, \label{eq:wang}
\end{align}
where $(a)_k=a(a+1)\cdots (a+k-1)$ is the Pochhammer symbol.
This also extends a conjectural result of the author and Schlosser \cite[Conjecture 6.2]{GS}.

In this paper, we shall give a $q$-analogue of \eqref{eq:wang} as follows:

\begin{thm}\label{thm:main-2}
Let $n>1$ be an odd integer. Then
\begin{align}
&\sum_{k=0}^{n-1}[3k-1]\frac{(q;q^2)_k (q^{-1};q^2)_k^2}{(q;q)_k^2 (q^2;q^2)_k} q^{(3k-k^2)/2} \notag\\[5pt]
&\quad\equiv [n]q^{-(n+1)/2}-(1+q)[n]^3+\frac{(n^2-1)(1-q)^2}{24}[n]^3q^{-(n+1)/2}  \pmod{[n]\Phi_n(q)^3}. \label{eq:wang-q}
\end{align}
\end{thm}

As before, letting $n=p>3$ be a prime and taking $q\to 1$ in \eqref{eq:wang-q}, we are led to \eqref{eq:guo-1}.
Moreover, it follows from \eqref{eq:guo-q} that, for any prime $p>3$ and $r\geqslant 2$,
\begin{align*}
\sum_{k=0}^{p^r-1} (3k-1)\frac{(\frac{1}{2})_k (-\frac{1}{2})_k^2 }{k!^3}4^k \equiv p^r \pmod{p^{r+3} }.
\end{align*}

We shall prove Theorems \ref{thm:main-1} and \ref{thm:main-2} by making use of the $q$-WZ method \cite{WZ1,WZ2} and the following
$q$-supercongruence: for odd $n$,
\begin{align}
&\sum_{k=0}^{n-1}[3k+1]\frac{(q;q^2)_k^3 q^{-{k+1\choose 2} } }{(q;q)_k^2 (q^2;q^2)_k} \notag\\[5pt]
&\qquad
\equiv q^{(1-n)/2}[n]+\frac{(n^2-1)(1-q)^2}{24}q^{(1-n)/2}[n]^3 \pmod{[n]\Phi_n(q)^3}.  \label{eq:q-div}
\end{align}
This $q$-supercongruence was originally conjectured in \cite{Guo-div} and recently proved in \cite{Guo-mod4}
with the help of the ``creative microscoping" method \cite{GuoZu}
and the Chinese reminder theorem. It is easy to see that \eqref{eq:div-1} follows from \eqref{eq:q-div} by taking
$n=p$ and $q\to 1$. The $n=p^r$ being an odd prime power and $q\to 1$ case of \eqref{eq:q-div} was conjectured by Z.-W. Sun \cite{Sun4}.
We point out that some other interesting $q$-supercongruences were given in
\cite{GG,Guo-aam,Guo-jmaa,Guo-a2,Guo-div,Guo-mod4,GL18,GS0,GS,GuoZu3,LP,NP,Straub,WN,WY0,WY,Zu19}.

The paper is organized as follows. We first give two lemmas in the next section.
The proofs of Theorems \ref{thm:main-1} and \ref{thm:main-2} will be given in Sections 3 and 4, respectively.
Two more similar $q$-supercongruences are given in Section 5. Finally, we propose a related
open problem in Section 6.

\section{Two lemmas}
In this section we give three simple $q$-congruences.
The first one may be deemed a $q$-analogue of Fermat's little theorem $2^{p-1}\equiv 1\pmod{p}$
for any odd prime $p$. The third one is a generalization of a recent result of Wang and Ni \cite[Lemma 2.2]{WN}.

\begin{lem}Let $n>1$ be an odd integer. Then
\begin{align}
(-q;q)_{n-1}\equiv 1\pmod{\Phi_n(q)}.  \label{eq:q-Fermat}
\end{align}
\end{lem}

\begin{proof}
 It is well known that
\begin{align}
\frac{x^n-1}{x-1}=\prod_{k=1}^{n-1}(x-\zeta^k),  \label{eq:fracxp-1}
\end{align}
where $\zeta$ is an $n$-th primitive root of unity. Letting $x=-1$ in \eqref{eq:fracxp-1}, we get
$(-\zeta;\zeta)_{n-1}=1$, which is equivalent to \eqref{eq:q-Fermat}.
\end{proof}

\begin{lem}Let $n>1$ be an odd integer. Then
\begin{align}
\frac{(q;q^2)_{n}}{(1-q)(q;q)_{n-1}} &\equiv [n]\pmod{[n]\Phi_n(q)}, \label{eq:mod-n} \\[5pt]
\frac{(q;q^2)_{n-1}}{(q;q)_{n-1}} &\equiv -[n]q\pmod{[n]\Phi_n(q)}. \label{eq:mod-n-new}
\end{align}
\end{lem}

\begin{proof}It is easy to see that
\begin{align}
\frac{(q;q^2)_{n}}{(1-q)(q;q)_{n-1}}
=[n]{2n\brack n}\frac{1}{(-q;q)_n},  \label{eq:mod-n-1}
\end{align}
and
\begin{align}
{2n\brack n}\frac{1}{(-q;q)_n}&=\frac{(q;q^2)_{(n-1)/2}(q^{n+2};q^2)_{(n-1)/2}}{(q;q)_{n-1}}  \notag \\[5pt]
&\equiv \frac{(q;q^2)_{(n-1)/2}(q^{2};q^2)_{(n-1)/2}}{(q;q)_{n-1}}=1\pmod{\Phi_n(q)}  \label{eq:mod-n-2}
\end{align}
in view of $q^n\equiv 1\pmod{\Phi_n(q)}$.
Since ${2n\brack n}$ is a polynomial in $q$ and $[n]$ is relatively prime to $(-q;q)_{n}$ for odd $n$,
the $q$-congruence \eqref{eq:mod-n} immediately follows from  \eqref{eq:mod-n-1} and \eqref{eq:mod-n-2}.

Observing that
\begin{align*}
\frac{(q;q^2)_{n-1}}{(q;q)_{n-1}}
&=\frac{1}{[2n-1]}\frac{(q;q^2)_{n}}{(1-q)(q;q)_{n-1}}, \\[5pt]
[2n-1]&=(1-q^{2n-1})/(1-q)\equiv -q^{-1}\pmod{\Phi_n(q)},
\end{align*}
and $[2n-1]$ and $[n]$ are relatively prime polynomials in $q$, we deduce \eqref{eq:mod-n-new} from \eqref{eq:mod-n}.
\end{proof}

\section{Proof of Theorem \ref{thm:main-1}}
Define two functions $F(n,k)$ and $G(n,k)$ as follows:
\begin{align}
F(n,k) &=[3n+2k+1]\frac{(q;q^2)_{n}(q^{2k+1};q^2)_{n}^2 q^{-{n+1\choose 2}-(2n+1)k} }{(q;q)_{n}^2 (q^2;q^2)_{n}}, \label{eq:fnk}\\[5pt]
G(n,k) &=-\frac{(1+q^{n+2k-1})(q;q^2)_{n}(q^{2k+1};q^2)_{n-1}^2 q^{-{n\choose 2}-(2n-1)k} }{(1-q)(q;q)_{n-1}^2 (q^2;q^2)_{n-1}}, \label{eq:gnk}
\end{align}
where we have assumed that $1/(q^2;q^2)_{n}=0$ for any negative integer $n$.
It is easy to check that
\begin{align}
F(n,k-1)-F(n,k)=G(n+1,k)-G(n,k).  \label{eq:fnk-gnk}
\end{align}
That is, the functions $F(n,k)$ and $G(n,k)$ form a $q$-WZ pair.

We now let $m>1$ be an odd integer.
Summing \eqref{eq:fnk-gnk} over $n$ from $0$ to $m-1$, we obtain
\begin{align}
\sum_{n=0}^{m-1}F(n,k-1)-\sum_{n=0}^{m-1}F(n,k)=G(m,k)-G(0,k)=G(m,k).  \label{eq:fnk-gn0-00}
\end{align}
In light of \eqref{eq:q-Fermat} and \eqref{eq:mod-n}, we have
\begin{align}
G(m,1) &=-\frac{(1+q^{m+1})(q;q^2)_{m}(q^3;q^2)_{m-1}^2 q^{-{m\choose 2}-2m+1} }{(1-q)(q;q)_{m-1}^3 (-q;q)_{m-1}}  \notag\\[5pt]
&=-\frac{(1+q^{m+1})(q;q^2)_{m}^3 q^{-{m\choose 2}-2m+1} }{(1-q)^3(q;q)_{m-1}^3 (-q;q)_{m-1}}   \notag \\[5pt]
&\equiv -(1+q)q[m]^3 \pmod{[m]^3\Phi_m(q)},  \label{eq:mod-nnnphi}
\end{align}
since $q^m\equiv 1\pmod{\Phi_m(q)}$. Combining \eqref{eq:fnk-gn0-00} and \eqref{eq:mod-nnnphi},
we conclude that
\begin{align}
\sum_{n=0}^{m-1}F(n,1)\equiv \sum_{n=0}^{m-1}F(n,0)+(1+q)q[m]^3  \pmod{[m]^3\Phi_m(q)}.  \label{eq:reslut-0}
\end{align}
It is easy to see that
\begin{align}
\sum_{n=0}^{m-1}F(n,1)
&=\sum_{n=0}^{m-1}
[3n+3]\frac{(q;q^2)_{n}(q^3;q^2)_{n}^2 q^{-{n+1\choose 2}-(2n+1)} }{(q;q)_{n}^2 (q^2;q^2)_{n}} \notag\\[5pt]
&=\sum_{n=1}^{m}
[3n]\frac{(q;q^2)_{n-1}(q^3;q^2)_{n-1}^2 q^{-{n\choose 2}-(2n-1)} }{(q;q)_{n-1}^2 (q^2;q^2)_{n-1}} \notag \\[5pt]
&=\sum_{n=1}^{m}\frac{[3n][2n][n]^2}{[2n-1](-q;q)_n^4}{2n\brack n}^3 q^{-(n^2+3n)/2+1}.  \label{eq:reslut-1}
\end{align}
On the other hand, by \eqref{eq:q-div} we have
\begin{align}
\sum_{n=0}^{m-1}F(n,0)
&=\sum_{n=0}^{m-1}[3n+1]\frac{(q;q^2)_n^3 q^{-{n+1\choose 2} } }{(q;q)_n^2 (q^2;q^2)_n}\notag \\[5pt]
&\equiv q^{(1-m)/2}[m]+\frac{(m^2-1)(1-q)^2}{24}q^{(1-m)/2}[m]^3 \pmod{[m]\Phi_m(q)^3}. \label{eq:reslut-2}
\end{align}
Substituting \eqref{eq:reslut-1} and  \eqref{eq:reslut-2} into \eqref{eq:reslut-0}, and noticing the $m$-th summand on the right-hand side of \eqref{eq:reslut-1} is
congruent to $0$ modulo $[m]^4$, we are led to \eqref{eq:guo-q} with $n\mapsto m$ differing only by a factor $q$.

\medskip
\noindent{\it Remark.} Usually the basic hypergeometric functions satisfying the condition
\begin{align*}
F(n,k+1)-F(n,k)=G(n+1,k)-G(n,k)
\end{align*}
are called a $q$-WZ pair. It is also reasonable to call the basic hypergeometric functions satisfying \eqref{eq:fnk-gnk} a $q$-WZ pair (see Zudilin \cite{Zudilin}).
The $q$-WZ pair in the proof of Theorem~\ref{thm:main-1} was found by the author \cite{Guo-div} in his proof
of a weaker form of \eqref{eq:q-div} modulo $[n]\Phi_n(q)^2$.
But the corresponding WZ pair (the limiting case $q\to 1$) was first given by Guillera and Zudilin \cite{GuZu}
in their proof of \eqref{eq:div-1}.

\section{Proof of Theorem \ref{thm:main-2}}

Let the functions $F(n,k)$ and $G(n,k)$ be given by \eqref{eq:fnk} and \eqref{eq:gnk}, respectively.
Again, let $m$ be an odd integer greater than $1$.
In view of \eqref{eq:q-Fermat}, \eqref{eq:mod-n}, \eqref{eq:mod-n-new} and $q^m\equiv 1\pmod{\Phi_m(q)}$, we have
\begin{align}
G(m,0) &=-\frac{(1+q^{m-1})(q;q^2)_{m}(q;q^2)_{m-1}^2 q^{-{m\choose 2}} }{(1-q)(q;q)_{m-1}^3 (-q;q)_{m-1}}  \notag\\[5pt]
&\equiv -(1+q^{-1})q^2[m]^3 \pmod{[m]^3\Phi_m(q)}.  \label{eq:mod-nnnphi-2}
\end{align}
It follows from \eqref{eq:fnk-gn0-00} and \eqref{eq:mod-nnnphi-2} that
\begin{align}
\sum_{n=0}^{m-1}F(n,-1)\equiv \sum_{n=0}^{m-1}F(n,0)-(1+q)q[m]^3  \pmod{[m]^3\Phi_m(q)}.  \label{eq:reslut-4}
\end{align}
By the definition of \eqref{eq:fnk}, we have
\begin{align}
\sum_{n=0}^{m-1}F(n,-1)
&=\sum_{n=0}^{m-1}
[3n-1]\frac{(q;q^2)_{n}(q^{-1};q^2)_{n}^2 q^{-{n+1\choose 2}+(2n+1)} }{(q;q)_{n}^2 (q^2;q^2)_{n}}.   \label{eq:reslut-5}
\end{align}
Finally, substituting \eqref{eq:reslut-2} and \eqref{eq:reslut-5} into \eqref{eq:reslut-4}, and dividing both sides by $q$,
we arrive at \eqref{eq:guo-q} with $n\mapsto m$.

\section{More similar $q$-supercongruences}
From \eqref{eq:q-div} and \eqref{eq:fnk-gn0-00} we can deduce more $q$-supercongruences  besides \eqref{eq:guo-q} and \eqref{eq:wang-q}.
Here we give two such examples.
\begin{thm}\label{thm:main-3}
Let $n>3$ be an odd integer. Then, modulo $[n]\Phi_n(q)^3$,
\begin{align}
&\sum_{k=0}^{n-1}[3k+5]\frac{(q;q^2)_k (q^{5};q^2)_k^2}{(q;q)_k^2 (q^2;q^2)_k} q^{-(k^2+9k)/2} \notag\\[5pt]
&\quad\equiv [n]q^{(5-n)/2}+(1+q)q^3[n]^3+\frac{(1+q^3)q^4}{(1+q+q^2)^2}[n]^3+\frac{(n^2-1)(1-q)^2}{24}[n]^3q^{(5-n)/2}. \label{eq:guo-q2}
\end{align}
\end{thm}

\begin{proof}Let $m>3$ be an odd integer. Then
\begin{align}
\sum_{n=0}^{m-1}F(n,2)=\sum_{k=0}^{m-1}[3k+5]\frac{(q;q^2)_k (q^{5};q^2)_k^2}{(q;q)_k^2 (q^2;q^2)_k} q^{-(k^2+9k)/2-2}.  \label{eq;more-1}
\end{align}
By \eqref{eq:fnk-gn0-00}, we get
\begin{align}
\sum_{n=0}^{m-1}F(n,2)= \sum_{n=0}^{m-1}F(n,0)-G(m,1)-G(m,2).\label{eq;more-2}
\end{align}
Moreover, in view of \eqref{eq:mod-n}, we have
\begin{align}
G(m,2) &=-\frac{(1+q^{m+3})(q;q^2)_{m}(q^5;q^2)_{m-1}^2 q^{-{m\choose 2}-4m+2} }{(1-q)(q;q)_{m-1}^3 (-q;q)_{m-1}}  \notag\\[5pt]
&=-\frac{(1-q^{2m+1})^2(1+q^{m+3})(q;q^2)_{m}^3 q^{-{m\choose 2}-4m+2} }{(1-q^3)^2(1-q)^3(q;q)_{m-1}^3 (-q;q)_{m-1}}   \notag \\[5pt]
&\equiv -\frac{(1+q^3)q^2}{(1+q+q^2)^2}[m]^3 \pmod{[m]\Phi_m(q)^3},
\label{eq:mod-m2}
\end{align}
Substituting \eqref{eq:mod-nnnphi}, \eqref{eq:reslut-2}, \eqref{eq;more-1} and \eqref{eq:mod-m2} into \eqref{eq;more-2},
we arrive at \eqref{eq:guo-q2} with $n\mapsto m$ differing only by a factor $q^{-2}$.
\end{proof}

\begin{thm}\label{thm:main-4}
Let $n>3$ be an odd integer. Then, modulo $[n]\Phi_n(q)^3$,
\begin{align}
&\sum_{k=0}^{n-1}[3k-3]\frac{(q;q^2)_k (q^{-3};q^2)_k^2}{(q;q)_k^2 (q^2;q^2)_k} q^{(7k-k^2)/2} \notag\\[5pt]
&\quad\equiv [n]q^{-(n+3)/2}-\frac{1+q}{q}[n]^3-\frac{1+q^3}{(1+q+q^2)^2}[n]^3+\frac{(n^2-1)(1-q)^2}{24}[n]^3q^{-(n+3)/2}. \label{eq:guo-q3}
\end{align}
\end{thm}

\begin{proof}The proof is similar to that of Theorem \ref{thm:main-3}. This time we need to use
\begin{align*}
\sum_{n=0}^{m-1}F(n,-2)= \sum_{n=0}^{m-1}F(n,0)+G(m,0)+G(m,-1),
\end{align*}
and
\begin{align*}
G(m,-1) &=-\frac{(1+q^{m-3})(q;q^2)_{m}(q^{-1};q^2)_{m-1}^2 q^{-{m\choose 2}+2m-1} }{(1-q)(q;q)_{m-1}^3 (-q;q)_{m-1}}  \notag\\[5pt]
&\equiv -\frac{(1+q^3)q^2}{(1+q+q^2)^2}[m]^3 \pmod{[m]\Phi_m(q)^3}.
\end{align*}
\end{proof}

\section{An open problem}
For Sun's divisibility result \eqref{eq:sunby}, we found that the following stronger version holds:
$$
4n{2n\choose n}\Bigg| \sum_{k=0}^{n-1} \frac{6k^4}{2k-1}{2k\choose k}^316^{n-k-1}.
$$
Furthermore, we believe that the following $q$-version should be true.
\begin{conj}Let $n$ be a positive integer. Then
\begin{align*}
\frac{1}{(1+q)^3[2n+1]{2n\brack n}}{}\sum_{k=1}^{n}\frac{[3k][2k][k]^2(-q;q)_n^4}{[2k-1](-q;q)_k^4}{2k\brack k}^3 q^{-(k^2+3k)/2}
\end{align*}
is a Laurent polynomial in $q$.
\end{conj}

\vskip 2mm \noindent{\bf Acknowledgments.}
This work was partially supported by the National Natural Science Foundation of China (grant 11771175).

\end{document}